\documentclass[letterpaper]{amsart}

\usepackage{amssymb}
\usepackage{amsthm}
\usepackage{amsmath}
\usepackage{enumitem}

\usepackage{euscript}
\usepackage{mathrsfs}

\usepackage{cancel}
\usepackage{soul}

\usepackage[normalem]{ulem}

\usepackage[foot]{amsaddr}

\usepackage[usenames]{xcolor}

\usepackage{hyperref}
 
\usepackage{tikz-cd}

\theoremstyle{plain}
\newtheorem{proposition}{Proposition}[section] 
\newtheorem{theorem}[proposition]{Theorem}
\newtheorem{lemma}[proposition]{Lemma}  
\newtheorem{corollary}[proposition]{Corollary}
\theoremstyle{definition}

\theoremstyle{remark}
\newtheorem{remark}[proposition]{Remark}

\DeclareMathOperator{\Isom}{\mathsf{Isom}}

\DeclareMathOperator{\supp}{supp}

\DeclareMathOperator{\dist}{d}

\DeclareMathOperator{\Fc}{\mathcal{F}}

\DeclareMathOperator{\Nc}{\mathcal{N}}

\DeclareMathOperator{\Tc}{\mathcal{T}}

\DeclareMathOperator{\Hb}{\mathbb{H}}

\DeclareMathOperator{\Nb}{\mathbb{N}}

\DeclareMathOperator{\Rb}{\mathbb{R}}
\DeclareMathOperator{\Sb}{\mathbb{S}}

\DeclareMathOperator{\Zb}{\mathbb{Z}}

\DeclareMathOperator{\Gsf}{\mathsf{G}}

\newcommand{\abs}[1]{\left|#1\right|}

\newcommand{\ip}[1]{\left\langle #1\right\rangle}

\newcommand{\Ga}{\Gamma}

\newcommand{\Hsf}{\mathsf{H}}

\newcommand{\R}{\mathbb{R}}

\newcommand{\ba}{\backslash}

\newcommand{\Mod}{\mathsf{Mod}}

\newcommand{\PMF}{\mathcal{PMF}}   

\newcommand{\UE}{\mathcal{UE}}

\newcommand{\Ic}{\mathcal{I}}
\newcommand{\QG}{\operatorname{QG}}

\begin{document}

\title{Determining subgroups via stationary measures}

\author[Kim]{Dongryul M. Kim}
\email{dongryul.kim97@gmail.com}
\address{Department of Mathematics, Yale University, USA}

\author[Zimmer]{Andrew Zimmer}
\email{amzimmer2@wisc.edu}
\address{Department of Mathematics, University of Wisconsin-Madison, USA}

\date{\today}

 \subjclass[2020]{}

\begin{abstract}

In this paper, we consider random walks on the isometry groups of general metric spaces.  Under some mild conditions, we show that if two non-elementary random walks on a discrete subgroup of the isometry group have non-singular stationary measures, then subgroups generated by the random walks are commensurable. This result in particular applies to Gromov hyperbolic spaces and Teichm\"uller spaces. As a specific application, we prove singularity between stationary measures associated to random walks on different fiber subgroups of the fundamental group of a hyperbolic 3-manifold fibering over the circle.

\end{abstract}

\maketitle

\section{Introduction}

Given a group $\Gsf$ and subgroups $\Hsf_1, \Hsf_2 < \Gsf$, we say that $\Hsf_1$ and $\Hsf_2$ are \emph{commensurable} if their intersection $\Hsf_1 \cap \Hsf_2$ is of finite index in both $\Hsf_1$ and $\Hsf_2$. Susskind and Swarup studied the commensurability of  two subgroups of a Kleinian group in terms of their limit sets. More precisely, they proved the following rigidity theorem.

\begin{theorem}[Susskind--Swarup \cite{SS_limit}] \label{thm:SS}
Suppose $\Gsf < \Isom^+(\Hb^n)$ is a discrete subgroup and $\Hsf_1, \Hsf_2 < \Gsf$ are non-elementary geometrically finite subgroups. If the limit sets of $\Hsf_1$ and $\Hsf_2$ in $\partial  \Hb^{n}$ are the same, then $\Hsf_1$ and $\Hsf_2$ are commensurable.
\end{theorem}

When $n = 3$, Anderson \cite{Anderson_intersections} and Yang--Jiang \cite{YJ_limit} relaxed the hypothesis in Theorem \ref{thm:SS}, using Canary's work on tame hyperbolic 3-manifolds \cite{Canary_ends} and the tameness conjecture (established by Agol \cite{Agol_Tameness} and Calegari--Gabai \cite{CG_Tameness}).

On the other hand, the commensurability rigidity as in Theorem \ref{thm:SS} is not true in general, e.g. a normal subgroup in a discrete group must have the same limit set as the entire group. Indeed, as a consequence of the virtual Haken conjecture proved by Agol \cite{Agol_Haken}, any closed hyperbolic 3-manifold has a finite cover fibering over the circle (\cite{Wise_fibering_announcement}, \cite{Wise_fibering}), and there is a plethora of different fibrations over the circle, as parametrized by Thurston's fibered cones \cite{Thurston_norm} (see also Fried's cones \cite{Fried_cone}). All such fibers give rise to surface subgroups contained in a cocompact lattice of $\Isom^+(\Hb^3)$ with the full limit set $\partial \Hb^3$. See Section \ref{subsec:CT} for detailed discussions.

Nevertheless, in this paper, we extend this commensurability rigidity to general subgroups, by shifting the perspective to considering
\begin{center}
\emph{random walks on subgroups and stationary measures.}
\end{center}
Noting that the limit set of a discrete subgroup in Theorem \ref{thm:SS} is the set of all accumulation points of its orbit, this new viewpoint is about the accumulation \emph{along a random trajectory}.

We now present our setup more precisely. Given a metric space $(X, \dist_X)$ and a countable group $\Gsf$ acting on $X$ by isometries, the random walk induced by a probability measure $\mathsf{m}$ on $\Gsf$ is given by 
$$
\omega_n := g_1 \cdots g_n \in \Gsf 
$$
where the $g_i$'s are independent identically distributed elements of $\Isom(X)$ each with distribution $\mathsf{m}$.

A bordification $\overline{X}$ of $X$ is a Hausdorff and second countable topological space to which $X$ is embedded as an open dense subset, such that the $\Gsf$-action on $X$ continuously extends to $\overline{X}$. An example of a bordification is the Gromov compactification of a proper Gromov hyperbolic metric space. 
Fixing a basepoint $o \in X$, the \emph{hitting measure} $\nu$ on $\overline{X}$ for the random walk is defined as follows: for a Borel subset $E \subset \overline{X}$, 
$$
\nu(E) := \operatorname{Prob} \left( \lim_{n \to + \infty} \omega_n o \text{ exists and is in } E \right).
$$
While $\nu$ is not always a probability measure, in the settings we consider it is indeed a probability measure supported on the boundary $\partial X := \overline{X} \smallsetminus X$. Note that the Markov property of the random walk implies that $\nu$ is \emph{$\mathsf{m}$-stationary}, i.e., $\mathsf{m} * \nu = \nu$.

We say that $\mathsf{m}$ has \emph{finite first moment} for $\dist_X$ if
$$
\mathbb{E} \left[ \dist_X(o, g o) \right] = \sum_{g \in \Gsf}  \dist_X(o, go)  \mathsf{m}(g) < + \infty.
$$
Throughout the paper, we consider two random walks on subgroups $\Hsf_1, \Hsf_2 < \Gsf$ with finite first moments for $\dist_X$. Particular examples we study include:
\begin{enumerate}
   \item $X$ is a geodesic Gromov hyperbolic space, $\Gsf < \Isom(X)$ acts metrically properly on $X$, and $\Hsf_1, \Hsf_2 < \Gsf$ are non-elementary subgroups.
   \item $X = \Hb^3$ is hyperbolic 3-space, $\Gsf < \Isom^+(\Hb^3)$ is the fundamental group of a hyperbolic 3-manifold fibering over the circle, and $\Hsf_1, \Hsf_2 < \Gsf$ are fiber subgroups. Along similar lines, $X = \Gsf$ is a hyperbolic free-by-cyclic group and $\Hsf_1, \Hsf_2 < \Gsf$ are free fiber subgroups.
   \item $X = \Tc(S)$ is the Teichm\"uller space of a closed surface $S$ with genus at least two and $\Hsf_1, \Hsf_2 < \Mod(S)$ are non-elementary subgroups of the mapping class group of $S$. 
   \end{enumerate}
In the above settings, each random walk has a unique stationary measure on an appropriate boundary, and is equal to the hitting measure. We will show that the non-singularity between stationary measures for random walks on $\Hsf_1$ and $\Hsf_2$ implies the commensurability of $\Hsf_1$ and $\Hsf_2$. Since the stationary measures are supported on the limit sets of $\Hsf_1$ and $\Hsf_2$, this extends the rigidity result of Susskind--Swarup (Theorem \ref{thm:SS}) to broader classes of groups using random walks.

We will prove a general statement  for a bordification of a geodesic metric space under certain hypotheses in Theorem \ref{thm:main universal}, and then deduce results in the introduction from it.

\subsection{Isometries on Gromov hyperbolic spaces}

Let $(X, \dist_X)$ be a geodesic Gromov hyperbolic space with the Gromov boundary $\partial  X$. We allow $X$ to be non-proper or non-separable, and hence $\partial X$ may be non-compact or not second countable. 

A countable subgroup $\Gsf < \Isom(X)$ of isometries is \emph{non-elementary} if it contains two loxodromic elements that fix disjoint pairs of points in $\partial  X$  and \emph{acts metrically properly} if  
$$
\# \{g \in \Gsf : g B \cap B \neq \emptyset \} < + \infty
$$
for any bounded subset $B \subset X$.

Maher--Tiozzo proved that for a probability measure $\mathsf{m}$ whose support generates a non-elementary subgroup $\Gsf$ as a group, there exists a unique $\mathsf{m}$-stationary measure $\nu$ on $\partial  X$ and is the same as the hitting measure for the random walk on $\Gsf$ induced by $\mathsf{m}$ \cite{MaherTiozzo2018}. 
Maher--Tiozzo assume that $X$ is separable, but as observed by Gruber--Sisto--Tessera~\cite[Remark 4]{GST2020} this assumption is not necessary for their results.
When $X$ is proper, this is due to Kaimanovich \cite{Kaimanovich2000}. 

Via stationary measures, we detect subgroups up to commensurability. 

\begin{theorem} \label{thm:GH}
Suppose $\Gsf < \Isom(X)$ is countable and acts metrically properly on $X$. Let $\Hsf_1, \Hsf_2 < \Gsf$ be non-elementary subgroups. For $j=1,2$ assume
\begin{itemize}
\item $\mathsf{m}_j$ is a probability measure on $\Hsf_j$ with finite first moment for $\dist_X$,
\item $\Hsf_j$ is generated by the support of $\mathsf{m}_j$, and 
\item $\nu_j$ is the $\mathsf{m}_j$-stationary measure on $\partial  X$.
\end{itemize} 
If  $\nu_1$ and $ \nu_2$ are not singular, then $\Hsf_1$ and $ \Hsf_2$ are commensurable. 
\end{theorem}

\begin{remark}\label{rmk:separable} As described by Gruber--Sisto--Tessera~\cite[Remark 4]{GST2020}, it is possible to construct an isometric action of $\Gsf$ on a separable geodesic  Gromov hyperbolic space $X'$ and a $\Gsf$-equivariant quasi-isometric embedding $X' \rightarrow X$. Using this, it suffices to prove Theorem~\ref{thm:GH} in the separable case. 
\end{remark}

As we will see in Proposition \ref{prop:normal subgps} and Proposition \ref{prop:fibration}, the moment condition is necessary.

As a corollary, we obtain the following singularity of stationary measures for the case that $X = \Gsf = \Hsf_2$ is a hyperbolic group.

\begin{corollary} \label{cor:hypgp}
   Suppose $\Gsf$ is a hyperbolic group and $\Hsf < \Gsf$ is a non-elementary subgroup of infinite index. Assume  respectively that 
   \begin{itemize}
      \item $\mathsf{m}_{\Gsf}$ and $\mathsf{m}_{\Hsf}$ are probability measures on $\Gsf$ and $\Hsf$ with finite first moments for a word metric on $\Gsf$,
      \item $\Gsf$ and $\Hsf$ are generated by the supports of $\mathsf{m}_{\Gsf}$ and $\mathsf{m}_{\Hsf}$, and 
      \item $\nu_{\Gsf}$ and $\nu_{\Hsf}$ are the $\mathsf{m}_G$-stationary and $\mathsf{m}_{\Hsf}$-stationary measures on $\partial \Gsf$.
   \end{itemize}
   Then $\nu_{\Gsf}$ and $\nu_{\Hsf}$ are mutually singular, i.e.,
   $$
   \nu_{\Gsf} \perp \nu_{\Hsf}.
   $$ 
\end{corollary}

The same statement holds when $\Gsf$ is a relatively hyperbolic group, replacing the Gromov boundary $\partial \Gsf$ above with the Bowditch boundary of $\Gsf$, and the word metric on $\Gsf$ with the metric on a Gromov model for $\Gsf$.

\subsection{Fibrations of hyperbolic 3-manifolds and Cannon--Thurston maps} \label{subsec:CT}

We present a different formulation of Theorem \ref{thm:GH} for some special cases. Suppose   a closed hyperbolic 3-manifold  $M$ admits a fibration
$$
S \to M \to \Sb^1
$$
over the circle with a fiber $S \subset M$. We simply call $M$ a fibered hyperbolic 3-manifold. Such $M$ has infinitely many different fibrations, parametrized by Thurston's fibered cones \cite{Thurston_norm}. It follows from Theorem \ref{thm:GH} that they can all be  distinguished by random walks and stationary measures.

\begin{corollary} \label{cor:fibration}
   Suppose $M$ is a fibered hyperbolic 3-manifold and $S_1, S_2 \subset M$ are fibers of two different fibrations of $M$ over the circle. For $j = 1, 2$ assume
\begin{itemize}
\item $\mathsf{m}_j$ is a probability measure on $\pi_1(S_j)$ with finite first moment for a word metric on $\pi_1(M)$,
\item $\pi_1(S_j)$ is generated by the support of $\mathsf{m}_j$, and 
\item $\nu_j$ is the $\mathsf{m}_j$-stationary measure on $\partial  \pi_1(M)$.
\end{itemize} 
Then $\nu_1$ and $\nu_2$ are mutually singular, i.e.,
$$
\nu_1 \perp \nu_2.
$$
\end{corollary}

Since $M$ is a closed hyperbolic 3-manifold, $\partial  \pi_1(M)$ can be identified with $\partial  \Hb^3$ in Corollary \ref{cor:fibration}.

For a fibered hyperbolic 3-manifold $M$ with a fiber $S \subset M$, we can regard $\pi_1(S)$ and $\pi_1(M)$ as discrete subgroups of $\Isom^+(\Hb^2)$ and $\Isom^+(\Hb^3)$ respectively. Then the inclusion $S \subset M$ induces a $\pi_1(S)$-equivariant embedding $\Hb^2 \to \Hb^3$. In \cite{CT_map}, Cannon and Thurston showed that this embedding continuously extends to a space-filling curve $\partial  \Hb^2 \to \partial  \Hb^3$, which is now called the \emph{Cannon--Thurston map} for the fibration $S \to M \to \Sb^1$.

\begin{corollary} \label{cor:CT}
   Suppose $M$ is a fibered hyperbolic 3-manifold with a fiber $S \subset M$ and the associated Cannon--Thurston map $f : \partial  \Hb^2 \to \partial  \Hb^3$. Assume respectively that 
\begin{itemize}
\item $\mathsf{m}_S$ and $\mathsf{m}_M$ are probability measures on $\pi_1(S)$ and $\pi_1(M)$ with finite first moments for the metric on $\Hb^3$,
\item $\pi_1(S)$ and $\pi_1(M)$ are  generated by the supports of $\mathsf{m}_S$ and $\mathsf{m}_M$, and 
\item $\nu_S$ and $\nu_{M}$ are the $\mathsf{m}_S$-stationary measure on $\partial  \Hb^2$ and the $\mathsf{m}_M$-stationary measure on $\partial  \Hb^3$.
\end{itemize} 
Then $f_* \nu_S$ and $\nu_M$ are mutually singular, i.e., 
$$
f_* \nu_S \perp \nu_M.
$$
\end{corollary}

Note that since $\pi_1(S)$ acts cocompactly on $\Hb^2$, the moment condition on $\mathsf{m}_S$ for the metric on $\Hb^3$ is weaker than the one for the metric on $\Hb^2$.

\begin{remark} 
   Since $\pi_1(S)$ and $\pi_1(M)$ can be regarded as cocompact lattices in $\Isom^+(\Hb^2)$ and $\Isom^+(\Hb^3)$, it follows from the work of Lyons--Sullivan \cite{LS_RW}  that Lebesgue measures on $\partial \Hb^2$ and $\partial \Hb^3$ are respectively stationary measures for random walks on $\pi_1(S)$ and $\pi_1(M)$ with finite exponential moments (see also Ballmann--Ledrappier \cite{BL_harmonic}). Therefore, the same singularity results as in Corollary \ref{cor:fibration} and Corollary \ref{cor:CT} hold  after replacing stationary measures with the Lebesgue measure on $\partial \Hb^3$ or the pushforward of the Lebesgue measure on $\partial \Hb^2$ through Cannon--Thurston map. The singularity between Lebesgue measures under the Cannon--Thurston map was first proved by Tukia \cite{Tukia1989} as a generalization of Mostow's rigidity theorem.

   We also refer to the work of Connell--Muchnik (\cite{CM_Harmonic_PS}, \cite{CM_Harmonicity_Gibbs}) 
   for general quasi-convex groups of isometries on $\operatorname{CAT}(-1)$-spaces and random walks on them whose stationary measures are quasi-conformal measures, or more general Gibbs measures.
\end{remark} 

\begin{remark}
   \
   \begin{enumerate}
      \item The singularity among stationary measures and Lebesgue measures through the Cannon--Thurston map was first proved by Gadre--Maher--Pfaff--Uyanik in \cite{GMPU_CT}, under conditions of  finite exponential moments and groups being generated by the supports as semigroups. They also proved quantitative results on quasi-geodesics. We relax the moment condition to finite first moment and the semigroup condition to a subgroup condition in Corollary~\ref{cor:CT}. 
      
      \item The notion of Cannon--Thurston map was generalized further by Mj \cite{Mj_CT} to hyperbolic groups and their normal subgroups with hyperbolic quotients. Analogous rigidity results for those generalized Cannon--Thurston maps can also be deduced from Theorem \ref{thm:GH} or Corollary \ref{cor:hypgp}. We refer to the work of Kapovich--Lustig \cite{KL_CT} for an explicit description of Cannon--Thurston maps for free-by-cyclic groups with hyperbolic iwip monodromies.
   \end{enumerate}
   
\end{remark}

\subsection{Mapping class groups and Teichm\"uller spaces}

Let $S$ be a closed connected orientable surface of genus at least two. The Teichm\"uller space $\Tc(S)$ of $S$ is the space of all marked hyperbolic structures on $S$, and it admits a natural metric called the Teichm\"uller metric $\dist_{\Tc}$ which is proper and geodesic. 
Thurston compactified the Teichm\"uller space by the space $\PMF$ of projective measured foliations on $S$ \cite{Thurston_geometry_dynamics}. This is now referred to as Thurston's compactification of $\Tc(S)$, and $\PMF$ is also called Thurston's boundary of $\Tc(S)$. 

The mapping class group $\Mod(S)$ of the surface $S$ is the group of isotopy classes of orientation-preserving homeomorphisms on $S$. The natural $\Mod(S)$-action on $\Tc(S)$ is proper and by isometries, and in fact $\Mod(S)$ is more or less the full isometry group of $(\Tc(S), \dist_{\Tc})$ as shown by Royden \cite{Royden_aut} and by Earle and Kra \cite{EK_holomorphic}, \cite{EK_isometries} (see also \cite{Ivanov_isometries}). Thurston also showed that the $\Mod(S)$-action on $\Tc(S)$ continuously extends to the action on Thurston's compactification $\Tc(S) \cup \PMF$.
A subgroup $\mathsf{H} < \Mod(S)$ is called \emph{non-elementary} if $\mathsf{H} $ contains two pseudo-Anosov mapping classes with distinct pairs of  invariant projective measured foliations.

In \cite{KM_MCG}, Kaimanovich and Masur showed that for a probability measure $\mathsf{m}$ on a non-elementary subgroup $\mathsf{H}  < \Mod(S)$ such that the support of $\mathsf{m}$ generates $\Hsf$ as a group, there exists a unique $\mathsf{m}$-stationary measure $\nu$ on $\PMF$, and is equal to the hitting measure of the random walk induced by $\mathsf{m}$. We show that stationary measures determine subgroups up to commensurability.

\begin{theorem} \label{thm:MCG}
   Suppose $\Hsf_1, \Hsf_2 < \Mod(S)$ are non-elementary subgroups. For $j=1,2$ assume
\begin{itemize} 
\item $\mathsf{m}_j$ is a probability measure on $\Hsf_j$ with finite first moment for $\dist_{\Tc}$,
\item $\Hsf_j$ is generated by the support of $\mathsf{m}_j$ as a group, and 
\item $\nu_j$ is the $\mathsf{m}_j$-stationary measure on $\PMF$.
\end{itemize} 
If  $\nu_1$ and $ \nu_2$ are not singular, then $\Hsf_1$ and $ \Hsf_2$ are commensurable. 
\end{theorem}

\begin{remark}
   In the forthcoming work, Eskin--Mirzakhani--Rafi \cite{EMR} show that the Lebesgue measure on $\PMF$ is a stationary measure for some random walk on $\Mod(S)$ with finite first moment for the \emph{Teichm\"uller metric} $\dist_{\Tc}$.
   Together with this, Theorem \ref{thm:MCG} implies the singularity of the Lebesgue measure on $\PMF$ and the stationary measure of the random walk on an infinite-index subgroup of $\Mod(S)$ with finite first moment for $\dist_{\Tc}$. Previously, Gadre--Maher--Tiozzo \cite{GMT_Teichmueller} proved singularity of stationary measures and the Lebesgue measure for random walks whose step distribution has finite first moment for the \emph{word metric} on $\Mod(S)$. In the case where the step distribution is supported on an infinite-index subgroup of  $\Mod(S)$, Theorem~\ref{thm:MCG} relaxes the moment condition to finite first moment for the \emph{Teichm\"uller metric}.\end{remark}

As a special example, let $\Ic < \Mod(S)$ be the \emph{Torelli group}, which consists of mapping classes acting trivially on the first homology group $H_1(S)$. 
As $\Ic$ is the kernel of the symplectic representation 
$$
\Mod(S) \twoheadrightarrow \mathsf{Sp}(2g, \Zb)
$$
where $g$ is the genus of $S$, the Torelli group $\Ic$ is an infinite-index normal subgroup of $\Mod(S)$. Moreover, one can deduce from Thurston's construction of pseudo-Anosov mapping classes  \cite{Thurston_geometry_dynamics} that $\Ic$ is non-elementary.

Hence, $\Ic$ is not commensurable to $\Mod(S)$ while its action on $\PMF$ is not dynamically distinguishable from that of $\Mod(S)$, i.e., both act minimally on $\PMF$ (\cite{FLP_Thurston}, \cite{MP_dynamics}). On the other hand, Theorem \ref{thm:MCG} implies that stationary measures are distinguished.

\begin{corollary}
   Suppose respectively that 
   \begin{itemize}
      \item $\mathsf{m}_{\Ic}$ and $\mathsf{m}$ are probability measures on $\Ic$ and $\Mod(S)$ with finite first moments for $\dist_{\Tc}$,
      \item $\Ic$ and $\Mod(S)$ are generated by the supports of $\mathsf{m}_{\Ic}$ and $\mathsf{m}$, and
      \item $\nu_{\Ic}$ and $\nu$ are the $\mathsf{m}_{\Ic}$-stationary and $\mathsf{m}$-stationary measures on $\PMF$.
   \end{itemize}
   Then $\nu_{\Ic}$ and $\nu$ are mutually singular, i.e., 
$$
\nu_{\Ic} \perp \nu.
$$
 \end{corollary}

\subsection{Organization} In Section \ref{sec:well-behaved}, we define the abstract setting for random walks and state the most general version of our commensurability rigidity. Sections \ref{sec:RW track QG}--\ref{sec:non-sing} are devoted to the proof of this general statement. In Section \ref{sec:RW track QG}, we prove that random walks track quasi-geodesics. The complementary phenomenon that quasi-geodesics track random walks is proved in Section \ref{sec:QG track RW}. In Section \ref{sec:non-sing}, we relate those trackings to stationary measures on quotients of groups. The necessity of the moment condition is discussed in Section \ref{sec:normal}, where we present examples that realize stationary measures of random walks on ambient groups by random walks on infinite-index normal subgroups.

\subsection*{Acknowledgements} 
We thank Yair Minsky for asking Kim about the singularity of stationary measures associated to random walks on different fiber subgroups (as in Corollary \ref{cor:fibration}). We also thank Inhyeok Choi, David Fisher, Caglar Uyanik, and the referee for useful comments. Kim expresses his special gratitude to his Ph.D. advisor Hee Oh for her encouragement and guidance. Kim also thanks the University of Wisconsin--Madison for hospitality during a visit in October 2025. 

Zimmer was partially supported by a Sloan research fellowship and grants DMS-2105580 and DMS-2452068 from the National Science Foundation.

\section{Well-behaved random walks and universal rigidity theorem}  \label{sec:well-behaved}

In this section, we introduce the abstract setup we consider throughout the paper and state our most general version of the rigidity theorem (Theorem \ref{thm:main universal}), from which all results in the introduction follow. This abstract setup is a modification of one considered by Tiozzo  \cite[Section 2]{Tiozzo_sublinear}.

\subsection{The abstract setup and main result}\label{sec:standing assumptions}

Let $(X, \dist_X)$ be a geodesic metric space and $\Gsf$ be a countable group acting by isometries on $X$. 
A \emph{bordification} $\overline{X}$ of $X$ is a Hausdorff and second countable topological space such that $X$ is homeomorphic to an open dense subset of $\overline{X}$ and that the $\Gsf$-action on $X$ continuously extends to $\overline{X}$, regarding $X$ as a subset of $\overline{X}$. We denote by $\partial X := \overline{X} \smallsetminus X$ the \emph{boundary} of $X$. 

For $a \ge 1$ and $K \ge 0$, a map $\sigma : \Rb \to X$ or its image is called a (bi-infinite) \emph{$(a, K)$-quasi-geodesic} if for all $t, s \in \Rb$,
$$
\frac{1}{a} \abs{t - s} - K \le \dist_X( \sigma(t), \sigma(s)) \le a \abs{t - s} + K.
$$
In forgetting a parametrization, we keep its orientation so that the image of a quasi-geodesic in $X$ comes with an orientation.

Let $\QG(X)$ denote the set of oriented non-parametrized quasi-geodesics in $X$ and let $ \mathscr{P}( \QG(X))$ denote the power set of $\QG (X)$. 
Suppose $P : \partial X \times \partial X \rightarrow~\mathscr{P}(\QG (X))$ is a $\Gsf$-equivariant map with the property that there exist $a \geq 1$ and $K \geq 0$ such that for each $(y^-, y^+) \in \partial X \times \partial X$, either $P(y^-, y^+)$ is empty or every element of $P(y^-,y^+)$ can be parametrized to be a $(a, K)$-quasi-geodesic. 

Fixing a basepoint $o \in X$, define the map $D: \partial X \times \partial X \rightarrow [0,+\infty]$ by 
$$
D(y^-,y^+) = \sup_{\sigma \in P(y^-,y^+)} \dist_X(o, \sigma) 
$$
(when $P(y^-,y^+) = \emptyset$, we define $D (y^-,y^+) = +\infty$).

\subsection{Random walks} Suppose $\Gsf$, $X$, $o \in X$, $P : \partial X \times \partial X \rightarrow \mathscr{P}(\QG(X))$, and $D :~\partial X \times \partial X \to \Rb$ are as in the previous section.

Let $\mathsf{m}$ be a probability measure on $\Gsf$. We consider the product space $(\Gsf^{\Zb}, \mathsf{m}^{\Zb})$ and denote each of its elements by $\mathbf{g} := (\dots, g_{-1}, g_0, g_1, g_2, \dots)$.
We often use the shift map $S : \Gsf^{\Zb} \to \Gsf^{\Zb}$, which is defined by $S((g_n)) = (g_{n+1})$. More precisely, for $\mathbf{g} = (g_n) \in \Gsf^{\Zb}$, the $n$-th component of $S(\mathbf{g})$ is $g_{n+1}$. The shift map preserves the measure $\mathsf{m}^{\Zb}$, and moreover is ergodic with respect to $\mathsf{m}^{\Zb}$.

We call $\mathsf{m}$ \emph{well-behaved with respect to $\overline{X}$ and $P$} if the following holds.
\begin{enumerate}[label=(W\arabic*)]
\item\label{item:limits exist} For $\mathsf{m}^{\Zb}$-a.e. $\mathbf{g} \in \Gsf^{\Zb}$, the limits 
$$
\zeta( \mathbf{g}) : =  \lim_{n \to + \infty} g_1 \cdots g_n o \in \partial X \quad \text{and} \quad \hat{\zeta}( \mathbf{g}) : =  \lim_{n \to + \infty} g_0^{-1} \cdots g_{-n}^{-1} o \in \partial X$$
exist. Let $\nu := \zeta_* \mathsf{m}^{\Zb}$ and $\hat \nu := \hat \zeta_* \mathsf{m}^{\Zb}$.

 \item\label{item:non-atomic} $\hat \nu$ is non-atomic.
\item\label{item:finite D} $D$ is finite $\hat \nu \otimes \nu$-a.e.
\item\label{item:bounded convergence} For $\hat \nu \otimes \nu$-a.e. $(y^-, y^+)$ and for every  parametrization $\sigma : \Rb \rightarrow X$ of an element of $P(y^- ,y^+)$,  
if $\{x_n \} \subset X$ and $\sup_{n \geq 0} \dist_X(x_n, \sigma(t_n)) < +\infty$ for some sequence $t_n \rightarrow \pm \infty$, then $x_n \rightarrow y^\pm$. 

\item\label{item:asymptotic} There exists $\kappa > 0$ such that for $\nu$-a.e. $y^+ \in \partial X$, if $y^-_1, y^-_2 \in \partial X \smallsetminus \{ y^+\}$ and $\sigma_j : \Rb \rightarrow X$  is a parametrization of an element of $P (y^-_j ,y^+)$ for $j=1,2$, then for some $t_1,t_2 \in \Rb$, 
$$
\sigma_1([t_1, +\infty)) \subset \Nc_\kappa\big(\sigma_2([t_2, +\infty)) \big) \quad \text{and} \quad \sigma_2([t_2, +\infty)) \subset \Nc_\kappa\big(\sigma_1([t_1, +\infty)) \big)
$$
where $\Nc_{\kappa}$ denotes the $\kappa$-neighborhood in $X$.

\end{enumerate}

Note that the two random variables $\zeta(\mathbf{g})$ and $\hat \zeta(\mathbf{g})$ are independent. We often consider the space of one-sided sequences $(\Gsf^{\Nb}, \mathsf{m}^{\Nb})$ and also denote each of its elements by $\mathbf{g} := (g_1, g_2, \dots)$. Then the $\mathsf{m}^{\Zb}$-a.e. defined map $\zeta : \Gsf^{\Zb} \to \partial X$ factors through $(\Gsf^{\Nb}, \mathsf{m}^{\Nb})$, i.e., for the projection $\Gsf^{\Zb} \to \Gsf^{\Nb}$, 
$$
(\dots, g_{-1}, g_0, g_1, g_2, \dots) \mapsto (g_1, g_2, \dots),
$$
 we have the following commutative diagram.
$$\begin{tikzcd}
\Gsf^{\Zb} \arrow{dr}{\zeta} \arrow{d} & \\
\Gsf^{\Nb} \arrow[dashed]{r} & \partial X
\end{tikzcd}
$$ 
Abusing notations, we denote the above measurable map $\Gsf^{\Nb} \to \partial X$ by $\zeta$, which is $\mathsf{m}^{\Nb}$-a.e. defined and satisfies $\nu = \zeta_* \mathsf{m}^{\Nb}$.

\subsection{Main result} We continue to suppose $\Gsf$, $X$, $o \in X$, and $P : \partial X \times \partial X \rightarrow \mathscr{P}(\QG(X))$ are as in Section~\ref{sec:standing assumptions} and  $\mathsf{m}$ is a probability measure on $\Gsf$. 

Recall that the \emph{first moment} of $\mathsf{m}$ for $\dist_X$ is 
$$
\sum_{g \in \Gsf} \dist_X(o, g o) \mathsf{m}(g) \in [0, + \infty].
$$
By Kingman's subadditive ergodic theorem, if $\mathsf{m}$ has finite first moment, then there exists $\ell(\mathsf{m}) \in [0, + \infty)$ so that for $\mathsf{m}^{\Zb}$-a.e. $\mathbf{g} = (g_n) \in \Gsf^{\Zb}$, 
$$
\ell(\mathsf{m}) = \lim_{n \to + \infty} \frac{1}{n} \dist_X(o, g_1 \cdots g_n o).
$$
The quantity $\ell(\mathsf{m})$ is called the \emph{linear drift} of $\mathsf{m}$.

The $\Gsf$-action on $X$ is called \emph{metrically proper} if for any bounded set $B \subset X$, the set $\{ g \in \Gsf : g B \cap B \neq \emptyset  \}$ is finite. 
Our main theorem of this paper is as follows.

\begin{theorem}[see Theorem~\ref{thm:comm} below] \label{thm:main universal}
Suppose the $\Gsf$-action on $X$ is metrically proper and $\mathsf{m}_1$, $\mathsf{m}_2$ are probability measures on $\Gsf$ which  have finite first moments for $\dist_X$, positive linear drifts, and are well-behaved with respect to $\overline{X}$ and $P$.

If their forward hitting measures $\zeta_* \mathsf{m}_1^{\Zb}$ and $\zeta_* \mathsf{m}_2^{\Zb}$ on $\partial X$ are not singular, then the subgroups
$$
\langle \supp \mathsf{m}_1 \rangle \quad \text{and} \quad \langle \supp \mathsf{m}_2 \rangle 
$$
generated by their supports are commensurable.
\end{theorem}

\subsection{Examples}\label{subsec:examples}

We now show that random walks on certain classes of metric spaces are well-behaved, with respect to natural bordifications, and have positive linear drifts. As a result, all the statements in the introduction will follow from Theorem~\ref{thm:main universal}.

\subsubsection{Isometry groups of separable Gromov hyperbolic spaces}

For $\delta \ge 0$, a geodesic metric space $(X, \dist_X)$ is called \emph{$\delta$-hyperbolic} if any geodesic triangle in $X$ is $\delta$-thin, i.e., for any geodesic triangle in $X$, one side is contained in the $\delta$-neighborhood of the union of the other two sides. The metric space $(X, \dist_X)$ is \emph{Gromov hyperbolic} if it is $\delta$-hyperbolic for some $\delta \ge 0$.

Let $(X, \dist_X)$ be a $\delta$-hyperbolic space.
For $o, x, y \in X$, the Gromov product of $x$ and $y$ with respect to $o$ is 
$$
\langle x, y \rangle_o := \frac{1}{2} \left( \dist_X(o, x) + \dist_X(o, y) - \dist_X(x, y) \right).
$$
This quantity measures the distance between $o$ and any geodesic segment between $x$ and $y$, up to an additive error depending only on $\delta$.

The \emph{Gromov boundary} of $X$ is defined as the space of equivalence classes of certain sequences in $X$:
$$
\partial X := \{ \{ x_n \} \subset X : \liminf_{n, m \to + \infty} \langle x_n, x_m \rangle_o = + \infty \} / \sim
$$
where $\{ x_n \} \sim \{ y_n \}$ if $\liminf_{n, m \to + \infty} \langle x_n, y_m \rangle_o = + \infty$. There are natural topologies on $\partial X$ and $X \cup \partial X$ so that if $X$ is separable, then $\overline{X} := X \cup \partial X$ is a bordification of $X$ 
(recall that we assume that a bordification is second countable). 
Further, if $X$ is proper, then $\partial X$ and $\overline{X}$ are compact.  See \cite{BH_book}, \cite{KB_boundaries}, \cite{DSU_hyperbolic} for more details.

There exists $K = K(\delta) \ge 0$ such that for any distinct $y^{\pm} \in \partial X$, there exists a $(1, K)$-quasi-geodesic $\sigma : \R \to X$ such that $\lim_{t \to \pm \infty} \sigma (t) = y^{\pm}$, see \cite[Remark 2.16]{KB_boundaries}. Hence with this $K$, we define the map $P: \partial X \times \partial X \to \mathscr{P}(\QG(X))$ as follows: for $(y^-, y^+) \in \partial X \times \partial X$,
$$
P(y^-, y^+) := \left\{ \sigma \in \QG(X) : \quad 
\begin{matrix}
\exists \text{ parametrization } \sigma : \Rb \to X \text{ s.t.}\\
\sigma \text{ is a } (1, K)\text{-quasi-geodesic and }
\displaystyle \lim_{t \to \pm \infty} \sigma(t) = y^{\pm}
\end{matrix}
\right\}.
$$
Then by the choice of $K \ge 0$, we have $P(y^-, y^+) \neq \emptyset$ if and only if $y^- \neq y^+$. Moreover, it is clear that $P$ is equivariant under the action of the isometry group of $X$.

Isometries of $X$ are classified into the following three types. For $g \in \Isom(X)$, either 
\begin{itemize}
   \item $g$ is \emph{elliptic}, i.e., $g$ has a bounded orbit in $X$,
   \item $g$ is \emph{parabolic}, i.e., $g$ is not elliptic and has a unique fixed point in $\partial X$, or 
   \item $g$ is \emph{loxodromic}, i.e., $g$ is not elliptic and has two distinct fixed points in $\partial X$, one is attracting and the other is repelling.
\end{itemize}
Two loxodromic elements $g, h \in \Isom(X)$ are \emph{independent} if they have disjoint sets of fixed points. A subgroup of $\Isom(X)$ is called \emph{non-elementary} if it contains two independent loxodromic isometries.

From now on, we further assume that $(X, \dist_X)$ is separable, but it may not be proper. 
Let $\Gsf < \Isom(X)$ be a non-elementary subgroup whose action on $X$ is metrically proper and suppose $\mathsf{m}$ is a probability measure on $\Gsf$ such that $\Gsf$ is generated by the support of $\mathsf{m}$ as a group. In this case, since the $\Gsf$-action on $X$ is metrically proper, the semigroup $\langle \supp \mathsf{m} \rangle_+$ generated by the support of $\mathsf{m}$ has an unbounded orbit. Hence, the non-elementary hypothesis on $\Gsf$ implies that $\langle \supp \mathsf{m} \rangle_+$ contains two independent loxodromic isometries \cite[ Theorem 6.2.3, Proposition 6.2.14]{DSU_hyperbolic}.

We now show that $\mathsf{m}$ is well-behaved with respect to $\overline{X}$ and $P$.  Since $\langle \supp \mathsf{m} \rangle_+$ contains two independent loxodromic isometries, a result of Maher--Tiozzo~\cite[Theorem 1.1]{MaherTiozzo2018} implies Property \ref{item:limits exist} and Property \ref{item:non-atomic}. While they further assumed  $\Gsf = \langle \supp \mathsf{m} \rangle_+$ throughout the paper, the proof of this statement works without the assumption that $\Gsf$ is generated by $\supp \mathsf{m}$ as a semigroup, as long as the semigroup $\langle \supp \mathsf{m} \rangle_+$ contains two independent loxodromic isometries. Then the non-atomicity in Property \ref{item:non-atomic} and the choice of $K \ge 0$ imply Property \ref{item:finite D}. Property \ref{item:bounded convergence} and Property \ref{item:asymptotic} are consequences of the Morse Lemma. 
 Therefore, $\mathsf{m}$ is well-behaved with respect to $\overline{X}$ and $P$.

Finally, when $\mathsf{m}$ has finite first moment, Gou\"ezel proved  $\ell(\mathsf{m}) > 0$ \cite[Theorem 1.1]{Gouezel_exp}. This was shown in \cite[Theorem 1.2]{MaherTiozzo2018} when $\Gsf = \langle \supp \mathsf{m} \rangle_+$. We note that finite first moment was only to ensure that $\ell(\mathsf{m})$ is well-defined. Replacing $\lim$ with $\liminf$ in defining $\ell(\mathsf{m})$, the positivity does not require any moment condition. 

Therefore, Theorem \ref{thm:main universal} applies in this setting 
and establishes Theorem~\ref{thm:GH} in the separable case. This implies the non-separable case, see Remark~\ref{rmk:separable}.

\subsubsection{Mapping class groups and Teichm\"uller spaces}

Let $S$ be a closed connected orientable surface of genus at least two so that $S$ admits a complete hyperbolic structure.  The mapping class group $\Mod(S)$ is the group of isotopy classes of orientation-preserving homeomorphisms of $S$, and the Nielsen--Thurston classification says that there are three categories of its elements (\cite{Nielsen_classification}, \cite{Thurston_geometry_dynamics}). That is, for $g \in \Mod(S)$,
\begin{itemize}
   \item $g$ is \emph{periodic}, i.e., $g$ has finite order.
   \item $g$ is \emph{reducible}, i.e., there exists a multicurve on $S$ invariant under $g$, up to isotopy.
   \item $g$ is \emph{pseudo-Anosov}, i.e., $g$ has a representative that preserves a pair of transverse (singular) measured foliations on $S$, stretching one and contracting the other. 
\end{itemize}
Note that a single mapping class can be both periodic and reducible, while pseudo-Anosov mapping classes are neither periodic nor reducible.

Let $\Tc(S)$ denote the  Teichm\"uller space of all marked hyperbolic structures on $S$ endowed with the Teichm\"uller metric. This is a proper and geodesic metric space. Moreover, the natural action of $\Mod(S)$ on $\Tc(S)$ is proper and by isometries. Let $\PMF$ denote the projective space of measured foliations on $S$. Thurston defined a compact topology on $\Tc(S) \cup \PMF$ to which the $\Mod(S)$-action on $\Tc(S)$ continuously extends, which is now referred to as Thurston's compactification \cite{Thurston_geometry_dynamics}. In particular, Thurston's compactification is a bordification of $\Tc(S)$.

In terms of the $\Mod(S)$-action on Thurston's compactification, the Nielsen--Thurston classification of mapping classes resembles the classification of isometries of Gromov hyperoblic spaces. Indeed, a transverse pair of measured foliations on $S$ invariant under a pseudo-Anosov mapping class is unique up to scaling, and hence it gives rise to a pair of points in $\PMF$ fixed by the pseudo-Anosov mapping class. Moreover, the stretching and contracting of the measured foliations imply the attracting and repelling of the fixed points in $\PMF$. This demonstrates that pseudo-Anosov mapping classes resemble loxodromic isometries on Gromov hyperbolic spaces. In this manner, we call two pseudo-Anosov mapping classes \emph{independent}, if they have  disjoint sets of fixed points in $\PMF$. Then a subgroup $\Gsf < \Mod(S)$ is \emph{non-elementary} if $\Gsf$ contains two independent pseudo-Anosov mapping classes. 

For the rest of this section, let $\Gsf < \Mod(S)$ be non-elementary, and suppose $\mathsf{m}$ is a probability measure on $\Gsf$ whose support generates $\Gsf$ as a group. 

We now prove that $\mathsf{m}$ is well-behaved with respect to Thurston's compactification $\Tc(S) \cup \PMF$ and appropriately defined map $P$. First, Property \ref{item:limits exist} and Property \ref{item:non-atomic} were proved by Kaimanovich--Masur \cite{KM_MCG}.

To define the map $P$, we note that Kaimanovich--Masur further showed that the hitting measures $\nu$ and $\hat \nu$ are supported on the subset $\UE \subset \PMF$ of uniquely ergodic measured foliations. More precisely, a measured foliation $\Fc$ on $S$ is \emph{uniquely ergodic} if it intersects all simple closed curves on $S$ and the only topologically equivalent measured foliations to $\Fc$ are multiples of $\Fc$. We also call its projective class uniquely ergodic, and $\UE \subset \PMF$ is the set of those projective classes. Since any two distinct measured foliations in $\UE$ are transverse, there exists a (unique) Teichm\"uller geodesic in $\Tc(S)$ having them as endpoints. Therefore, we can define the map $P : \PMF \times \PMF \to \mathscr{P}(\QG(\Tc(S)))$ by 
$$
P(y^-, y^+) = \{ \text{the Teichm\"uller geodesic from } y^- \text{ to } y^{+} \} \quad \text{if } y^{\pm} \in \UE \text{ are distinct,}
$$ and $P(y^-, y^+) = \emptyset$ otherwise. Then the map $P$ is $\Mod(S)$-equivariant.

Now based on the fact that $\nu$ and $\hat \nu$ are supported on $\UE$ and are non-atomic, Property \ref{item:finite D} follows.  Moreover, Property \ref{item:bounded convergence} follows from  \cite[Lemma 1.4.2]{KM_MCG}, since $\nu$ and $\hat \nu$ are supported on $\UE$. 
Property \ref{item:asymptotic} follows from the work of Masur \cite{Masur_unique}, together with that $\nu$ is supported on $\UE$. Therefore, $\mathsf{m}$ is well-behaved with respect to Thurston's compactification $\Tc(S) \cup \PMF$ and  $P$.

Finally, Tiozzo's sublinear geodesic tracking \cite{Tiozzo_sublinear} implies that if $\mathsf{m}$ has finite first moment, then $\ell(\mathsf{m}) > 0$. Hence, Theorem \ref{thm:main universal} applies.

\section{Random walks track quasi-geodesics}  \label{sec:RW track QG}

We continue to suppose $\Gsf$, $\overline{X}$, $o \in X$, $P : \partial X \times \partial X \rightarrow \mathscr{P}(\QG(X))$, and $D : \partial X \times \partial X \to \Rb$ are as in Section~\ref{sec:standing assumptions}. Let $a \geq 1$ and $K \geq 0$ be the constants in the definition of $P$.

In this section we observe that a generic random walk spends most of its time in a neighborhood of a quasi-geodesic ray. Analogous statements for different settings were proved in \cite[Theorem A bis]{Benard_RW}, \cite[Proof of Theroem 2.2.4]{KM_MCG}, for instance.

\begin{theorem} \label{thm:RW track G}
    Suppose  $\mathsf{m}$ is well-behaved with respect to $\overline{X}$ and $P$. Assume
    \begin{equation}\label{eqn:escape to infinity as}
    \lim_{n \rightarrow + \infty} \dist_X(g_1\cdots g_n o, o) =+\infty
    \end{equation}
    for $\mathsf{m}^{\Zb}$-a.e. $\mathbf{g} = (g_n) \in \Gsf^{\Zb}$ (e.g. $\mathsf{m}$  has finite first moment and positive linear drift).

For every $\epsilon > 0$ there exists $R > 0$ such that: For $\mathsf{m}^{\Zb}$-a.e. $\mathbf{g} = (g_n) \in \Gsf^{\Zb}$, if $\sigma : \Rb \rightarrow X$ is a $(a, K)$-quasi-geodesic parametrizing an element of $P(\hat{\zeta}(\mathbf{g}) ,\zeta( \mathbf{g}) )$, then 
$$
 \liminf_{N \to + \infty} \frac{1}{N} \# \left\{ 1 \le n \le N : \dist_X(g_1\cdots g_n o , \sigma|_{[0, +\infty)}) \leq R \right\} > 1-\epsilon .
   $$
\end{theorem}

\begin{remark} For the particular examples we consider in Section~\ref{subsec:examples}, it suffices to consider the case when $a=1$. \end{remark}

\begin{proof} 
By Property \ref{item:limits exist} and Property \ref{item:finite D}, for $\mathsf{m}^{\Zb}$-a.e. $\mathbf{g} = (g_n) \in \Gsf^{\Zb}$, the limits 
$$
\zeta(\mathbf{g}) = \lim_{n \to + \infty} g_1 \dots g_n o  \quad \text{and} \quad \hat \zeta(\mathbf{g}) := \lim_{n \to + \infty} g_0^{-1} g_{-1}^{-1} \dots g_{-n}^{-1} o 
$$
 exist in $\partial X$ and $D(\hat{\zeta} (\mathbf{g}), \zeta (\mathbf{g}))$ is finite. In particular, 
   $$
\lim_{R \rightarrow +\infty}    \mathsf{m}^{\Zb} \left( \left\{ \mathbf{g} \in \Gsf^{\Zb} : D (\hat{\zeta} (\mathbf{g}), \zeta (\mathbf{g}))  \le R \right\} \right)=1.
   $$
   So we can fix $R > 0$ where the set 
   $$
   A_R : =\left\{ \mathbf{g} \in \Gsf^{\Zb} : D (\hat{\zeta} (\mathbf{g}), \zeta (\mathbf{g}))  \le R \right\} 
   $$
   satisfies $\mathsf{m}^{\Zb}(A_R) > 1-\epsilon$. 
   
Note that for the shift map $S : \Gsf^{\Zb} \to \Gsf^{\Zb}$,
     $$
     \hat \zeta(S\mathbf{g}) = g_1^{-1} \hat \zeta (\mathbf{g}) \quad \text{and}  \quad \zeta (S \mathbf{g}) = g_1^{-1} \zeta (\mathbf{g}) \quad      \mathsf{m}^{\Zb}\text{-a.e.}
     $$
Since $P$ is $\Gsf$-equivariant,
   $$
 D\left( \hat \zeta (S^n \mathbf{g}), \zeta (S^n \mathbf{g}) \right) = \sup_{\sigma \in P(\hat{\zeta} (\mathbf{g}), \zeta (\mathbf{g}))} \dist_X( g_1 \cdots g_n o, \sigma).
   $$
   Since $S$ preserves $\mathsf{m}^{\Zb}$ and is ergodic, it follows from the Birkhoff ergodic theorem that for $\mathsf{m}^{\Zb}$-a.e. $\mathbf{g} = (g_n) \in \Gsf^{\Zb}$,
   \begin{align}
   \lim_{N \to + \infty}  & \frac{1}{N} \#  \left\{1 \le n \le N  :  \sup_{\sigma \in P(\hat{\zeta} (\mathbf{g}), \zeta (\mathbf{g}))} \dist_X( g_1 \cdots g_n o, \sigma) \le R \right\} \nonumber \\
   & = \lim_{N \to + \infty} \frac{1}{N} \sum_{n = 1}^N \mathbf{1}_{A_{R} }(S^n \mathbf{g})  = \mathsf{m}^{\Zb} \left( A_R  \right) > 1-\epsilon.\label{eqn:ae condition in first theorem}
   \end{align}

      Since $\hat \zeta (\mathbf{g})$ and $\zeta (\mathbf{g})$ are independent and $\hat \nu$ is non-atomic by Property \ref{item:non-atomic}, for $\mathsf{m}^{\Zb}$-a.e. $\mathbf{g} = (g_n) \in \Gsf^{\Zb}$ we have 
      \begin{equation}\label{eqn:limits points not equal in pm directions}
      \hat{\zeta} (\mathbf{g}) \neq \zeta (\mathbf{g}).
      \end{equation}
      
Now $\mathsf{m}^{\Zb}$-a.e. $\mathbf{g} = (g_n) \in \Gsf^{\Zb}$ satisfies Equations~\eqref{eqn:escape to infinity as}, ~\eqref{eqn:ae condition in first theorem}, and~\eqref{eqn:limits points not equal in pm directions}. Fix such  $\mathbf{g}$ and then fix a $(a, K)$-quasi-geodesic $\sigma : \Rb \rightarrow X$ parametrizing an element of $P ( \hat{\zeta} (\mathbf{g}) ,\zeta (\mathbf{g}) )$.

      Let 
   \begin{align*}
   I  & : = \left\{ n \in \Nb : \dist_X( g_1 \cdots g_n o, \sigma)\leq R \right\} = \{ n_1 < n_2 < \dots \}. 
   \end{align*}
   For each $n_j \in I$ fix $t_{n_j} \in \Rb$ with 
$$
\dist_X( g_1 \cdots g_{n_j} o, \sigma(t_j))\leq R.
$$
By Equation~\eqref{eqn:ae condition in first theorem},
$$
\liminf_{N \rightarrow + \infty} \frac{1}{N} \#(I \cap [1,N]) > 1-\epsilon, 
$$
so to finish the proof we need to show that $t_j \rightarrow +\infty$. By Equation~\eqref{eqn:escape to infinity as}, we have $\abs{t_j} \rightarrow + \infty$. Since $g_1 \cdots g_{n_j} o \rightarrow \zeta(\mathbf{g}) \in \partial X$ and  $\zeta(\mathbf{g}) \neq \hat\zeta(\mathbf{g})$ by Equation~\eqref{eqn:limits points not equal in pm directions}, Property~\ref{item:bounded convergence} implies that $t_j \rightarrow +\infty$. 
\end{proof}

\section{Quasi-geodesics track random walks} \label{sec:QG track RW}

We continue to suppose $\Gsf$, $\overline{X}$, $o \in X$, and $P : \partial X \times \partial X \rightarrow \mathscr{P}(\QG(X))$ are as in Section~\ref{sec:standing assumptions}. Let $a \geq 1$ and $K \geq 0$ be the constants in the definition of $P$. 

In this section, we prove the complementary statement of Theorem \ref{thm:RW track G} that a generic quasi-geodesic ray spends most of its time in a bounded neighborhood of a random walk.

For $R > 0$, let $\Nc_{R}(S)$ denote the $R$-neighborhood of a subset $S \subset X$.

\begin{theorem} \label{thm:G track RW}
   Suppose $\mathsf{m}$ has finite first moment, positive linear drift, and is well-behaved with respect to $\overline{X}$ and $P$.

   For every $\epsilon > 0$ there exists $R > 0$ such that: For $\mathsf{m}^{\Zb}$-a.e. $\mathbf{g} = (g_n) \in \Gsf^{\Zb}$, if $\sigma : \Rb \to X$ is a $(a, K)$-quasi-geodesic parametrizing an element of $P(\hat \zeta (\mathbf{g}), \zeta(\mathbf{g}))$, then
$$
\liminf_{T \rightarrow +\infty} \frac{1}{T} \abs{ \left\{ t \in [0,T] : \sigma(t) \in \Nc_R\left( \cup_{n \geq 1} g_1 \cdots g_n o\right)\right\} } > 1 - \epsilon .
$$
\end{theorem} 

\begin{remark} For the particular examples we consider in Section~\ref{subsec:examples}, it suffices to consider the case when $a=1$. \end{remark}

\begin{proof}   Fix $\epsilon_0 > 0$ such that 
    $1-2\epsilon_0 > 0$ and 
$$
\frac{2 a^4}{1-2\epsilon_0}  \frac{\ell(\mathsf{m})+\epsilon_0}{\ell(\mathsf{m})}\epsilon_0<\epsilon.
$$

By Theorem \ref{thm:RW track G}, we can fix $R_0 > 0$ such that for $\mathsf{m}^{\Zb}$-a.e. $\mathbf{g} = (g_n) \in \Gsf^{\Zb}$, if $\sigma : \Rb \to X$ is a $(a, K)$-quasi-geodesic parametrizing an element of $P(\hat \zeta(\mathbf{g}), \zeta(\mathbf{g}))$, then 
   $$
\liminf_{N \to + \infty} \frac{1}{N} \# \left\{ 1 \le n \le N : \dist_X(g_1\cdots g_n o , \sigma|_{[0, + \infty)}) \leq R_0 \right\} > 1-\epsilon_0.
   $$
   
By assumption, for $\mathsf{m}^{\Zb}$-a.e. $\mathbf{g} = (g_n) \in \Gsf^{\Zb}$, 
\begin{equation} \label{eqn:linear growth}
 \ell(\mathsf{m}) =\lim_{n \rightarrow + \infty} \frac{1}{n} \dist_X(o, g_1 \cdots g_n o) > 0.
\end{equation}
For each $k \in \Nb$, let
$$
A_k : = \left\{ \mathbf{g} = (g_n) \in \Gsf^{\Zb} : \dist_X( o, g_1\cdots g_n o) \leq (\ell(\mathsf{m})+\epsilon_0)n  \text{ for all } n \geq k \right\}. 
$$
Then
$$
\lim_{k \rightarrow + \infty} \mathsf{m}^{\Zb}( A_k) = 1.
$$
Hence we can fix $k \in \Nb$ such that $\mathsf{m}^{\Zb}( A_k) > 1-\epsilon_0$. Recall that the shift map $S : \Gsf^{\Zb} \rightarrow \Gsf^{\Zb}$ 
preserves $\mathsf{m}^{\Zb}$ and is ergodic. Then by the Birkhoff ergodic theorem, for $\mathsf{m}^{\Zb}$-a.e. $\mathbf{g} \in \Gsf^{\Zb}$, 
$$
\lim_{N \rightarrow + \infty} \frac{1}{N} \# \left\{1 \le  n \le N: S^{n} \mathbf{g} \in A_k \right\}= \lim_{N \rightarrow + \infty} \frac{1}{N}\sum_{n=1}^{N} \mathbf{1}_{A_k}( S^n \mathbf{g}) =  \mathsf{m}^{\Zb}( A_k) > 1-\epsilon_0.
$$
Thus for $\mathsf{m}^{\Zb}$-a.e. $\mathbf{g}  = (g_n) \in \Gsf^{\Zb}$, if $\sigma : \Rb \to X$ is a $(a, K)$-quasi-geodesic parametrizing an element of $P(\hat \zeta(\mathbf{g}), \zeta(\mathbf{g}))$, then  
\begin{equation}\label{eqn:defn of c}
\liminf_{N \to + \infty} \frac{1}{N} \# \left\{1 \le  n \le N: 
\begin{matrix}
S^{n} \mathbf{g} \in A_k \text{ and} \\
 \dist_X(g_1\cdots g_n o , \sigma|_{[0, + \infty)}) \leq R_0 
\end{matrix}\right\} \ge 1-2\epsilon_0.
\end{equation} 

Fix $\mathbf{g} = (g_n) \in \Gsf^{\Zb}$ and $\sigma : \Rb \to X$ such that 
Equations~ \eqref{eqn:linear growth} and \eqref{eqn:defn of c}  hold.
 Let 
$$
I_0 := \left\{ n \in \Nb  : S^{n} \mathbf{g} \in A_k \text{ and } \dist_X(g_1\cdots g_n o , \sigma|_{[0, + \infty)}) \leq R_0 \right\}
$$
and let $I := \{ n_1 < n_2 < \cdots\} \subset I_0$ be a maximal $k$-separated subset, i.e., $I$ is a maximal subset of $I_0$ such that $|n_i - n_j| \ge k$ for all distinct $n_i, n_j \in I_0$. Notice that if $n_{j+1} \ge n_j+2k$, then by maximality 
$$
n_j+k, n_j + k +1, \dots, n_{j+1} -k \notin I_0. 
$$
So 
$$
\#(I_0 \cap [1,N]) \leq N - \sum_{n_{j+1} \leq N, \ n_{j+1} - n_j \geq 2k}   n_{j+1}-n_j-2k+1.
$$
Hence 
\begin{equation} \label{eqn:maximality}
\limsup_{N \rightarrow + \infty} \frac{1}{N} \sum_{n_{j+1} \leq N, \ n_{j+1} - n_j \geq 2k}   n_{j+1}-n_j-2k+1 \leq 2\epsilon_0. 
\end{equation}

Fix $R > 0$ such that 
$$
2 \cdot \frac{R - R_0 - K}{a^2 ( \ell(\mathsf{m}) + \epsilon_0)} - \frac{2 R_0 + K}{\ell(\mathsf{m}) + \epsilon_0} \geq 2k.
$$
For each $j \in \Nb$, fix $T_j \ge 0$ such that  
$$
\dist_X(g_1 \cdots g_{n_j} o, \sigma(T_j)) \leq R_0.
$$
Let 
$$
\Omega:= \{ t \geq 0 : \sigma(t) \notin \Nc_R( \cup_{n \ge 1} g_1 \cdots g_n o)\}.
$$
For $J \ge 2$, if  $t \in \Omega \cap [T_1,T_J]$, then there exists $j \leq J-1$ with 
$$
T_j + \frac{R - R_0 - K}{a} \leq t \leq  T_{j+1}- \frac{R - R_0 - K}{a}.  
$$
Further, 
$$
t \in \Omega \cap [T_j, T_{j+1}] \subset \left[ T_j+ \frac{R - R_0 - K}{a} , T_{j+1} - \frac{R - R_0 -K}{a} \right]
$$
and so, using the fact that $S^{n_j}\mathbf{g} \in A_k$, we have 
$$\begin{aligned}
2 \cdot \frac{R - R_0 - K}{a} & \leq T_{j+1}-T_j \leq a \dist_X(\sigma(T_{j}), \sigma(T_{j+1})) + a K \\
& \le a \left( 2R_0 + \dist_X(g_1\cdots g_{n_j} o, g_1\cdots g_{n_{j+1}} o) \right) + a K \\
& \leq a ( 2R_0 + K) + a (\ell(\mathsf{m})+\epsilon_0)(n_{j+1}-n_j).
\end{aligned} 
$$
So, in this case, 
$$
n_{j+1} - n_j \ge  2 \cdot \frac{R - R_0 - K}{a^2 ( \ell(\mathsf{m}) + \epsilon_0)} - \frac{2 R_0 + K}{\ell(\mathsf{m}) + \epsilon_0} \geq 2k.
$$

Thus 
\begin{align*}
& {\rm Leb}(\Omega \cap [T_1,T_J]) \\
& \leq \sum_{ \substack{ j \leq J-1, \ T_j < T_{j+1} \\ \Omega \cap [T_j, T_{j+1}] \neq \emptyset } } T_{j+1} - T_j - 2 \cdot \frac{R - R_0 - K}{a} \\
& \leq \sum_{ \substack{ j \leq J-1, \ T_j < T_{j+1} \\ \Omega \cap [T_j, T_{j+1}] \neq \emptyset } } a ( 2R_0 + K) + a (\ell(\mathsf{m})+\epsilon_0)(n_{j+1}-n_j) - 2 \cdot \frac{R - R_0 - K}{a} \\
& \leq a (\ell(\mathsf{m})+\epsilon_0) \sum_{ n_{j+1} \leq n_{J}, \  n_{j+1} - n_j \geq 2k} n_{j+1}-n_j-2k.
\end{align*} 
Hence by Equation \eqref{eqn:maximality},
$$
\limsup_{J \rightarrow + \infty} \frac{1}{n_J} {\rm Leb}(\Omega \cap [0,T_J]) \leq 2a(\ell(\mathsf{m})+\epsilon_0)\epsilon_0.
$$
Since 
$$
\frac{\ell(\mathsf{m})}{a} \le \liminf_{J \to + \infty} \frac{T_J}{n_J} \le \limsup_{J \to + \infty} \frac{T_J}{n_J} \le a \ell(\mathsf{m}), 
$$ 
 we then have 
$$
\limsup_{J \rightarrow + \infty} \frac{1}{T_J} {\rm Leb}(\Omega \cap [0,T_J]) \leq 2a^2 \cdot  \frac{\ell(\mathsf{m})+\epsilon_0}{\ell(\mathsf{m})}\epsilon_0.
$$

To finish the proof it suffices to show that $\limsup_{J \rightarrow + \infty} \frac{T_{J+1}}{T_J} \leq \frac{a^2}{1-2\epsilon_0}$.  Indeed, then for any $T > 0$ there exists $J$ with $T_J \leq T \leq T_{J+1}$ and hence 
\begin{align*}
\limsup_{T \rightarrow + \infty} \frac{1}{T} {\rm Leb}(\Omega \cap [0,T])  & \leq \limsup_{J \rightarrow + \infty} \frac{1}{T_J} {\rm Leb}(\Omega \cap [0,T_{J+1}])\\
& \leq \frac{2 a^4}{1-2\epsilon_0}  \frac{\ell(\mathsf{m})+\epsilon_0}{\ell(\mathsf{m})}\epsilon_0<\epsilon.
\end{align*}

Suppose for a contradiction that $\limsup_{J \rightarrow + \infty} \frac{T_{J+1}}{T_J} > \frac{a^2}{1-2\epsilon_0}$. Then there exist $J_i \rightarrow \infty$ and $c > \frac{a^2}{1-2\epsilon_0}$ such that 
$$
\lim_{i \rightarrow + \infty} \frac{T_{J_i+1}}{T_{J_i}} = c,
$$
 where we a priori allow $c = + \infty$ with the convention $\frac{1}{+\infty} = 0$.
Then 
$$
\frac{c}{a^2} \le \liminf_{i \rightarrow + \infty} \frac{n_{J_i+1}}{n_{J_i}} \le \limsup_{i \rightarrow + \infty} \frac{n_{J_i+1}}{n_{J_i}} \le a^2 c.
$$
By maximality, we have 
$$
n_{J_i}+k, n_{J_i} + k +1, \dots, n_{J_i+1} -k \notin I_0.
$$
So 
\begin{align*}
1-2\epsilon_0&  \leq \liminf_{i \rightarrow + \infty} \frac{1}{n_{J_i+1}} \#(I_0 \cap [0,n_{J_i+1}]) \\
& \leq  \liminf_{i \rightarrow + \infty} \frac{n_{J_i+1} - ( n_{J_i+1} -n_{J_i} - 2k+1)}{n_{J_i+1}} \leq \frac{a^2}{c} < 1-2\epsilon_0.
\end{align*}
Thus we have a contradiction. 
\end{proof} 

\section{Non-singular stationary measures} \label{sec:non-sing}

 We are now ready to prove Theorem \ref{thm:main universal}, which we restate below in a slightly different format. 
 
 We continue to suppose $\Gsf$, $\overline{X}$, $o \in X$, and $P : \partial X \times \partial X \rightarrow \mathscr{P}(\QG(X))$ are as in Section~\ref{sec:standing assumptions}. Let $a \geq 1$ and $K \geq 0$ be the constants in the definition of $P$. 
     
\begin{theorem} \label{thm:comm} Assume
\begin{itemize}
\item  the $\Gsf$-action on $X$ is metrically proper,
\item  $\mathsf{m}_1$ and $\mathsf{m}_2$ are probability measures on $\Gsf$ which have finite first moments, positive linear drifts, and are well-behaved with respect to $\overline{X}$ and $P$, and 
\item for $j = 1, 2$, $\Hsf_j < \Gsf$ is the group generated by the support of $\mathsf{m}_j$ and $\nu_j := \zeta_* \mathsf{m}_j^{\Nb} = \zeta_* \mathsf{m}_j^{\Zb}$ denotes the forward hitting measure of the random walk induced by $\mathsf{m}_j$.
\end{itemize} 
If $\nu_1$ and $\nu_2$ are non-singular, then $\Hsf_1$ and $\Hsf_2$ are commensurable.
\end{theorem}

Using Theorem~\ref{thm:G track RW} we will show that a positive proportion of the random walks generated by  $\mathsf m_1$ stays in a bounded neighborhood of $\Hsf_2 o$ most of the time.

\begin{lemma} \label{lem:RW track RW}
With the hypothesis of Theorem~\ref{thm:comm}, for any $\epsilon > 0$ there exist a measurable subset $E \subset \Gsf^{\Nb}$ and $R > 0$ such that 
\begin{enumerate}
\item $\mathsf{m}_1^{\Nb}(E) > 0$ and
\item if $\mathbf{g} = (g_n) \in E$, then
$$
\liminf_{N \to + \infty} \frac{1}{N} \# \{ 1 \le n \le N : \dist_X(g_1 \cdots g_n o, \Hsf_2 o) < R \} > 1 - \epsilon.
$$
\end{enumerate}
\end{lemma}

Assuming Lemma \ref{lem:RW track RW} for a moment, we prove the theorem.

\begin{proof}[Proof of Theorem~\ref{thm:comm}]

By Lemma \ref{lem:RW track RW}, there exist a measurable subset $E \subset \Gsf^{\Nb}$ with $\mathsf{m}_1^{\Nb}(E) > 0$ and  $R > 0$ such that for $\mathbf{g} = (g_n) \in E$,
$$
\liminf_{N \to + \infty} \frac{1}{N} \# \{ 1 \le n \le N : \dist_X(g_1 \cdots g_n o, \Hsf_2 o) < R \} > 1/2.
$$
Since the $\Gsf$-action on $X$ is metrically proper, there exists a finite set $F \subset \Gsf$ such that $\{ g \in \Gsf : \dist_X (g o, \Hsf_2 o) < R \} \subset \Hsf_2 F$. Then for all $\mathbf{g} = (g_n) \in E$,
$$
\liminf_{N \to  + \infty} \frac{1}{N} \# \{ 1 \le n \le N : g_1 \cdots g_n \in \Hsf_2 F \} > 1/2.
$$

Let $\pi : \Gsf  \to \Hsf_2 \ba \Gsf$ denote the quotient map and let $\pi_* \mathsf{m}_1$ denote the pushforward of $\mathsf{m}_1$ on $\Hsf_2 \ba \Gsf$. One can see that for each $n \in \Nb$ and $g \in \Gsf$,
$$
\mathsf{m}_1^{*n}(\Hsf_2 g) = \left( (\pi_* \mathsf{m}_1) * \mathsf{m}_1^{*(n-1)} \right) (\Hsf_2 g),
$$
where on the left hand side we have $\Hsf_2 g \subset \Gsf$ and on the right hand side we have $\Hsf_2 g \in \Hsf_2 \ba \Gsf$. For each $N \in \Nb$, consider the probability measure 
$$
\mu_N :=  \frac{1}{N} \sum_{n = 1}^{N} (\pi_* \mathsf{m}_1) * \mathsf{m}_1^{*(n-1)}$$
on $\Hsf_2 \ba \Gsf$.  

By Fatou's lemma,
$$\begin{aligned}
\liminf_{ N \to + \infty} \frac{1}{N} \sum_{n = 1}^{N} \mathsf{m}_1^{*n}(\Hsf_2 F) & = \liminf_{ N \to + \infty} \frac{1}{N} \int \# \{ 1 \le n \le N : g_1 \cdots g_n \in \Hsf_2 F \} d \mathsf{m}_1^{\Nb}(\mathbf{g}) \\
& \ge  \int \liminf_{ N \to + \infty} \frac{1}{N}  \# \{ 1 \le n \le N : g_1 \cdots g_n \in \Hsf_2 F \} d \mathsf{m}_1^{\Nb}(\mathbf{g}) \\
& \ge \mathsf{m}_1^{\Nb}(E) / 2 > 0.
\end{aligned}
$$
Hence on the compact subset $\Hsf_2 F \subset \Hsf_2 \ba \Gsf$, the measure $\mu_N$ is uniformly bounded from below by $\mathsf{m}_1^{\Nb}(E)/3 > 0$, for all large $N \in \Nb$. 

Since $ \Hsf_2 \ba \Gsf$ is countable, we can fix a subsequence $\{\mu_{N_j}\}$ and a measure $\mu$ such that for any finite subset $F' \subset \Hsf_2 \ba \Gsf$, the sequence $\{\mu_{N_j}|_{F'}\}$ converges in the \mbox{weak-$*$} topology to $\mu|_{F'}$. 
Then $\mu$ is a finite non-zero measure on $\Hsf_2 \ba \Gsf$. By construction, the measure $\mu$ is $\mathsf{m}_1$-stationary, i.e., $\mu * \mathsf{m}_1 = \mu$. Further, since $\supp \mathsf{m}_1 \subset \Hsf_1$, the measure $\mu$ is supported on $\Hsf_2 \ba \Hsf_2 \Hsf_1$.
Let 
$$
\hat \Hsf_1 := \left\{ h \in \Hsf_1 : \mu(\Hsf_2 h) = \max_{g \in \Hsf_1} \mu(\Hsf_2 g) \right\}.
$$
Since $\mu$ is a finite non-zero measure,
$$
0 < \# \left( \Hsf_2 \ba \Hsf_2 \hat \Hsf_1 \right)< + \infty.
$$

Now for $h \in \hat \Hsf_1$,
$$
\mu(\Hsf_2 h) = (\mu * \mathsf{m}_1)(\Hsf_2 h) = \sum_{g \in \Gsf} \mu(\Hsf_2 h g^{-1}) \mathsf{m}_1(g).
$$
This implies
$$
\mu(\Hsf_2 h) = \mu(\Hsf_2 h g^{-1}) \quad \text{for all} \quad g \in \supp \mathsf{m}_1.
$$
In particular, 
$h \cdot \left(\supp \mathsf{m}_1\right)^{-1} \subset \hat \Hsf_1$.
Since this holds for any $h \in \hat \Hsf_1$,
$$
\Hsf_2 \ba \Hsf_2 \hat \Hsf_1 \left(\supp \mathsf{m}_1\right)^{-1} \subset \Hsf_2 \ba \Hsf_2 \hat \Hsf_1.
$$
Since $\Hsf_2 \ba \Hsf_2 \hat \Hsf_1$ is a finite set, this implies that 
$$
\Hsf_2 \ba \Hsf_2 \hat \Hsf_1 \left(\supp \mathsf{m}_1\right)^{-1} = \Hsf_2 \ba \Hsf_2 \hat \Hsf_1.
$$
Thus
$$
\Hsf_2 \ba \Hsf_2 \hat \Hsf_1 = \Hsf_2 \ba \Hsf_2 \hat \Hsf_1\left(\supp \mathsf{m}_1\right).
$$
Then, since $\langle \supp \mathsf{m}_1 \rangle = \Hsf_1$,
$$
\Hsf_2 \ba \Hsf_2 \Hsf_1 =  \Hsf_2 \ba \Hsf_2 \hat \Hsf_1.
$$
Therefore, $$
\# \left(\Hsf_2 \ba \Hsf_2 \Hsf_1\right) < + \infty,
$$
and hence $\Hsf_1 \cap \Hsf_2$ is a finite index subgroup of $\Hsf_1$. Switching $\Hsf_1$ and $\Hsf_2$, the same argument implies that $\Hsf_1 \cap \Hsf_2$ is a finite index subgroup of $\Hsf_2$ as well. 
\end{proof}

\subsection{Proof of Lemma \ref{lem:RW track RW}}

By Theorem \ref{thm:RW track G} applied to $\mathsf{m}_1$, there exists $R_1 = R_1(\epsilon)> 0$ such that for $\mathsf{m}_1^{\Zb}$-a.e. $\mathbf{g} = (g_n) \in \Gsf^{\Zb}$, if $\sigma_1 : \R \to X$  is a $(a, K)$-quasi-geodesic parametrizing an element of $P(\hat \zeta(\mathbf{g}), \zeta(\mathbf{g}))$, then 
   $$
   \liminf_{N \to + \infty} \frac{1}{N} \# \left\{ 1 \le n \le N : \dist_X(g_1\cdots g_n o , \sigma_1|_{[0, + \infty)}) \leq R_1 \right\} > 1-\epsilon/4.
   $$

   Since $\mathsf{m}_1$ has finite first moment and positive linear drift, there exists $k \in \Nb$ such that the set 
$$
A_k := \left\{ \mathbf{g} = (g_n) \in \Gsf^{\Zb} :  \dist_{X}(o, g_1 \cdots g_n o) >  2 R_1 (1 + a) + K \text{ for all } n \ge k \right\}
$$
satisfies
$$
\mathsf{m}_1^{\Zb}(A_k) > 1 - \epsilon/4.
$$
Recall that $S : (\Gsf^{\Zb}, \mathsf{m}_1^{\Zb}) \to (\Gsf^{\Zb}, \mathsf{m}_1^{\Zb})$ denotes the shift map. Similar to the proof of Theorem \ref{thm:G track RW}, we have from the Birkhoff ergodic theorem that for $\mathsf{m}_1^{\Zb}$-a.e. $\mathbf{g} = (g_n) \in \Gsf^{\Zb}$, if $\sigma_1 : \R \to X$ is a $(a, K)$-quasi-geodesic parametrizing an element of $P(\hat \zeta(\mathbf{g}), \zeta(\mathbf{g}))$, then 
\begin{equation} \label{eqn:RW track RW m1}
 \liminf_{N \to + \infty} \frac{1}{N} \# \left\{1 \le  n \le N: 
 \begin{matrix}
 S^{n} \mathbf{g} \in A_k \text{ and}\\
  \dist_X(g_1\cdots g_n o , \sigma_1 |_{[0, + \infty)}) \leq R_1
 \end{matrix} \right\} >
1-\epsilon/2.
\end{equation}

Now fix $\epsilon_1 > 0$ such that 
$$
\frac{ (2k+1)a\ell(\mathsf{m}_1) }{R_1} \cdot \epsilon_1 < \frac{\epsilon}{2}.
$$
By  Theorem \ref{thm:G track RW} applied to $\mathsf{m}_2$, there exists $R_2 = R_2(\epsilon_1) > 0$ such that for $\nu_2$-a.e. $  x \in \partial  X$, there exists $y \in \partial X$ so that if $\sigma_2 : \Rb \to X$  is a $(a, K)$-quasi-geodesic parametrizing an element of $P(y, x)$, then 
\begin{equation} \label{eqn:RW track RW m2}
\liminf_{T \rightarrow +\infty} \frac{1}{T} \abs{ \left\{ t \in [0,T] : \sigma_{2}(t) \in \Nc_{R_2}\left( \Hsf_2 o\right)\right\} } > 1 - \epsilon_1.
\end{equation}
By Property \ref{item:non-atomic}, we may assume $y \neq x$.

Let  $E' \subset \partial  X$ denote the set of points satisfying Equation \eqref{eqn:RW track RW m2}. Since $\nu_1 = \zeta_* \mathsf{m}_1^{\Zb}$ and $\nu_2$ are non-singular, we have 
 $\mathsf{m}_1^{\Zb}(\zeta^{-1} E') > 0$. Hence, there exists a subset $E  \subset \Gsf^{\Zb}$ such that $\mathsf{m}_1^{\Zb}(E) > 0$ and that for any $\mathbf{g} \in E$, $\mathbf{g}$ and $\zeta(\mathbf{g})$ satisfy Equations \eqref{eqn:RW track RW m1} and \eqref{eqn:RW track RW m2} respectively. We may further assume that 
$$\lim_{n \to + \infty} \frac{1}{n} \dist_X(o, g_1 \cdots g_n o) = \ell(\mathsf{m}_1) > 0
$$
for all $\mathbf{g} = (g_n) \in E$.

Then by Property \ref{item:asymptotic},  for any $\mathbf{g} = (g_n) \in E$ and any $(a, K)$-quasi-geodesic   $\sigma : \Rb \to X$ parametrizing an element of $P(\hat \zeta(\mathbf{g}), \zeta (\mathbf{g}))$, we have 
\begin{equation} \label{eqn:RW track RW nonsing0}
\liminf_{N \to + \infty} \frac{1}{N} \# \left\{1 \le  n \le N:
\begin{matrix}
    S^{n} \mathbf{g} \in A_k \text{ and} \\
     \dist_X(g_1\cdots g_n o , \sigma |_{[0, + \infty)}) \leq R_1
\end{matrix}  \right\} >
1-\epsilon/2
\end{equation}
and
\begin{equation} \label{eqn:RW track RW nonsing}
\liminf_{T \rightarrow +\infty} \frac{1}{T} \abs{ \left\{ t \in [0,T] : \sigma(t) \in \Nc_{R_2 + \kappa}\left( \Hsf_2 o\right)\right\} } > 1 - \epsilon_1.
\end{equation}
(Here $\kappa > 0$ is the number satisfying Property \ref{item:asymptotic}. Unfortunately, due to the large number of parameters, the very similar symbols $k$, $\kappa$, $K$ are used in this proof.)

We claim that $E$ and any 
$$
R > (1 + a) R_1 + K + R_2 + \kappa
$$
satisfy the lemma. Fix such $R$.

Fixing $\mathbf{g} = (g_n) \in E$ and $\sigma : \Rb \to X$ as above, let 
$$I_0 :=  \left\{n \in \Nb : S^{n} \mathbf{g} \in A_k \text{ and } \dist_X(g_1\cdots g_n o , \sigma|_{[0, + \infty)}) \leq R_1 \right\}.$$
Then let 
$$
I_0' : = \left\{ n \in I_0 : g_1 \cdots g_n o \notin \Nc_{R}(\Hsf_2 o) \right\}. 
$$
By Equation~\eqref{eqn:RW track RW nonsing0}, 
\begin{align*}
\liminf_{N \to + \infty} & \frac{1}{N} \# \{ 1 \le n \le N : \dist_X(g_1 \cdots g_n o, \Hsf_2 o) < R \} \\
& \geq \liminf_{N \to + \infty} \frac{1}{N} \#(I_0 \cap [1,N]) - \limsup_{N \to + \infty} \frac{1}{N} \#(I_0' \cap [1,N]) \\
& > 1-\epsilon/2 - \limsup_{N \to + \infty} \frac{1}{N} \#(I_0' \cap [1,N])
\end{align*}
and so it suffices to show that 
$$
 \limsup_{N \to + \infty} \frac{1}{N} \#(I_0' \cap [1,N]) \leq \epsilon/2. 
 $$
 This is clearly true if $I_0'$ is finite and so we can assume that $I_0'$ is infinite. 
 
Fix a maximal $k$-separated set $I \subset I_0'$, i.e. $\abs{n-m} \geq k$ for all distinct $m,n \in I$. Then by maximality, 
$$
I_0' \subset \bigcup_{n \in I} [n-k,n+k]
 $$
and so 
$$
 \limsup_{N \to + \infty} \frac{1}{N} \#(I_0' \cap [1,N]) \leq  \limsup_{N \to + \infty} \frac{2k+1}{N} \#(I \cap [1,N]).
 $$
 
 Enumerate $I=\{ n_1 < n_2 < \cdots\}$ and for each $j \in \Nb$ fix $t_j \in [0,+\infty)$ with 
 $$
 \dist_X(g_1 \cdots g_{n_j} o, \sigma(t_j)) \leq R_1. 
 $$
 Notice that 
 $$
 \sigma([t_j-R_1, t_j+R_1]) \subset \Nc_{R_1+a R_1 + K}( g_1 \cdots g_{n_j} o) 
 $$
 and so 
  $$
 \sigma([t_j-R_1, t_j+R_1]) \cap \Nc_{R_2+\kappa}( \Hsf_2o)=\emptyset.  
 $$
Further, if $n_i < n_j$, then 
$$
\dist_X(g_1 \cdots g_{n_i} o, g_1 \cdots g_{n_j} o) =\dist_X(o, g_{n_i+1} \cdots g_{n_j} o)> 2 R_1 (1 + a) + K
$$
since $S^{n_i}\mathbf g \in A_k$ and $I$ is $k$-separated. So 
$$
\abs{t_i-t_j} \geq \frac{1}{a} \dist_X( \sigma(t_i),\sigma(t_j)) - \frac{K}{a} > 2R_1
$$
and hence 
$$
[t_i-R_1, t_i+R_1] \cap [t_j-R_1, t_j+R_1] =\emptyset.
$$

Next let $T_N = \max \{ t_j : n_j \leq N\}$. Then the above implies that 
$$
2R_1 \cdot \#(I \cap [1,N]) \leq  \abs{ \left\{ t \in [-R_1,T_N+R_1] : \sigma(t) \notin \Nc_{R_2 + \kappa}\left( \Hsf_2 o\right)\right\} }.
$$
Notice that 
$$
t_j \leq a\dist_X(\sigma(t_j), \sigma(0)) + a K \leq a \dist_X( g_1 \cdots g_{n_j} o, \sigma(0)) + a R_1 + a K
$$
and so
$$
\limsup_{N \rightarrow + \infty} \frac{T_N}{N}  \leq a \ell(\mathsf{m}_1).
$$ 
Then Equation~\eqref{eqn:RW track RW nonsing} implies that 
$$
\limsup_{N \rightarrow + \infty}\frac{1}{N} \#(I \cap [1,N]) \leq \frac{  a \ell(\mathsf{m}_1) }{2R_1} \epsilon_1 
$$
and so 
$$
\limsup_{N \rightarrow + \infty}\frac{1}{N} \#(I_0' \cap [1,N]) \leq (2k+1)\frac{  a \ell(\mathsf{m}_1) }{2R_1} \epsilon_1 < \epsilon/2
$$
by our choice of $\epsilon_1$. 

Thus 
$$
\liminf_{N \to + \infty} \frac{1}{N} \# \{ 1 \le n \le N : \dist_X(g_1 \cdots g_n o, \Hsf_2 o) < R \} > 1 - \epsilon,
$$
as desired. Now replacing $E$ with its image under the projection $\Gsf^{\Zb} \to \Gsf^{\Nb}$ finishes the proof.
\qed

\section{Random walks on normal subgroups} \label{sec:normal}

In this section, suppose  $\Gsf$ is a  finitely generated group, $\Gsf$ acts by homeomorphisms on a compact metrizable space $Y$, and $\Hsf \triangleleft  \Gsf$ is a normal subgroup with $\Gsf / \Hsf \cong \Zb^k$ where $k \in \{1,2\}$. Also, let $\abs{\cdot}$ denote a word length on $\Gsf$ with respect to some fixed generating set.

Recall that a probability measure $\nu$ on $Y$ is \emph{$\mathsf m$-stationary} if 
$$
m * \nu = \nu.
$$
We call $\mathsf{m}$ \emph{symmetric} if $\mathsf{m}(g) = \mathsf{m}(g^{-1})$ for all $g \in \Gsf$.

\begin{proposition} \label{prop:normal subgps}
    Suppose $\mathsf{m}$ is a symmetric probability measure on $\Gsf$ whose support generates $\Gsf$ as a group and $\nu$ is a $\mathsf m$-stationary measure on $Y$.  If $\mathsf{m}$ has finite $k^{th}$ moment for $\abs{\cdot}$, i.e.,    
    $$
\sum_{g \in \Gsf} \abs{g}^k \mathsf{m}(g) < +\infty,
$$
then there exists a symmetric probability measure $\mathsf{m}'$ on $\Hsf$ whose support generates $\Hsf$ as a group and where $\mathsf{m}'*\nu=\nu$. 
\end{proposition} 

The rest of the section is devoted to the proof of the proposition. For $\mathbf{g} = (g_n) \in \Gsf^{\Nb}$, define the stopping time
$$
\tau(\mathbf{g}) := \inf \{ n \geq 1 : g_1 \cdots g_n \in \Hsf\}.
$$

\begin{lemma} \label{lem:finite stopping}
$\tau$ is finite $\mathsf{m}^{\Nb}$-a.e.
\end{lemma} 

\begin{proof} Let $\pi : \Gsf \rightarrow \Gsf / \Hsf$ be the quotient map. Then $\pi_* \mathsf{m}$ induces a symmetric random walk with finite $k^{th}$ moment on $\Gsf / \Hsf \cong \Zb^k$ and $\tau$ represents the first return time of this random walk to $0 \in \Zb^k$. Since $k \in \{1,2\}$, by \cite{CF_RW} this walk is recurrent, which implies that $\tau$ is finite $\mathsf{m}^{\Nb}$-a.e.
\end{proof} 

Next define $\xi : \Gsf^{\Nb} \rightarrow \Hsf$ by $\xi(\mathbf{g}) = g_1 \cdots g_{\tau(\mathbf{g})}$ and let $\mathsf{m}': = \xi_* \mathsf{m}^{\Nb}$. Since  $\mathsf{m}$ is a symmetric probability measure whose support generates $\Gsf$ as a group, $\mathsf{m}'$ is a symmetric probability measure whose support generates $\Hsf$ as a group.

\begin{lemma} 
 $\mathsf{m}'*\nu=\nu$.
\end{lemma} 

\begin{proof} 
Let ${\rm Prob}(Y)$ denote the space of probability measures on $Y$. By the martingale convergence theorem, there exists a measurable map $\mathbf{g} \in  \Gsf^{\Nb} \mapsto \nu_{\mathbf{g}} \in  {\rm Prob}(Y)$ so that:
\begin{enumerate}
\item  For $\mathsf{m}^{\Nb}$-a.e. $\mathbf{g} = (g_n) \in \Gsf^{\Nb}$,
$$
(g_1 \cdots g_n)_* \nu \rightarrow \nu_{\mathbf{g}}
$$ as $n \to + \infty$,
\item 
$\nu = \int \nu_{\mathbf{g}} d \mathsf{m}^{\Nb} (\mathbf{g})$
\end{enumerate} 
(see for instance ~\cite[Lemmas 2.17 and 2.19]{BQ_book}). Notice that (1) implies that 
$$
\nu_{( g_1,g_2, g_3, \dots)} = (g_1)_* \nu_{(g_2, g_3, \dots)}
$$
for $\mathsf{m}^{\Nb}$-a.e. $\mathbf{g} = (g_n) \in \Gsf^{\Nb}$.

For $n \in \Nb$, let $\pi_n : \Gsf^{\Nb} \rightarrow \Gsf^n$ be the projection onto the first $n$ factors and let 
$$
A_n : = \{ \pi_n(\mathbf{g}) : \tau(\mathbf{g}) = n\}.
$$
Then the sets $A_n \times \Gsf^{\Nb} \subset \Gsf^{\Nb}$ are disjoint and their union has full $\mathsf{m}^{\Nb}$-measure in $\Gsf^{\Nb}$ by Lemma \ref{lem:finite stopping}. So writing $\mathbf{h} = (h_n) \in \Gsf^{\Nb}$,
\begin{align*}
\mathsf{m}' * \nu  & = \sum_{h \in \Hsf} \mathsf{m}'(h) h_* \nu  = \int \int (h_1 \cdots h_{\tau(\mathbf{h})})_* \nu_{\mathbf{g}} d \mathsf{m}^{\Nb}(\mathbf{h}) d \mathsf{m}^{\Nb}(\mathbf{g}) \\
& = \int \int  \nu_{(h_1,\dots, h_{\tau(\mathbf{h})}, \mathbf{g})} d \mathsf{m}^{\Nb}(\mathbf{h}) d \mathsf{m}^{\Nb}(\mathbf{g}) \\
& = \int \sum_{n=1}^\infty \sum_{ (h_1,\dots, h_n) \in A_n}  \nu_{(h_1,\dots, h_n,\mathbf{g})} d\mathsf{m}(h_1) \cdots d\mathsf{m}(h_n) d \mathsf{m}^{\Nb}(\mathbf{g}) \\
& = \sum_{n=1}^\infty \int_{A_n \times \Gsf^{\Nb}}  \nu_{\mathbf{g}}d \mathsf{m}^{\Nb}(\mathbf{g})= \int \nu_{\mathbf{g}} d \mathsf{m}^{\Nb} (\mathbf{g})=\nu. \qedhere
\end{align*}
\end{proof}

\begin{remark}
   We remark that when $\Gsf / \Hsf \cong \Zb$ and the symmetric probability measure $\mathsf{m}$ has finite support, then the induced probability measure $\mathsf{m}'$ has finite $p^{th}$ moment for all $p < 1/2$, i.e., 
   $$
   \sum_{h \in \Hsf} \abs{h}^{p} \mathsf{m}'(h) < + \infty
   $$
   where $\abs{\cdot}$ is a word length on $\Gsf$ with respect to some fixed generating set.

   To see this, note that for $C_0 := \max_{g \in \supp \mathsf{m}} \abs{g} < + \infty$, we have $\abs{\xi(\mathbf{g})} \le C_0 \cdot \tau(\mathbf{g})$. Hence, 
   $$
   \sum_{h \in \Hsf} \abs{h}^{p} \mathsf{m}'(h) = \int \abs{\xi(\mathbf{g})}^p d \mathsf{m}^{\Nb}(\mathbf{g}) \le C_0^p \cdot \sum_{n = 1}^{\infty} n^p \mathsf{m}^{\Nb} \left( \left\{ \mathbf{g} \in \Gsf^{\Nb} : \tau(\mathbf{g}) = n \right\} \right). 
   $$
   So for some $C_1 > 0$,
   $$
\sum_{h \in \Hsf} \abs{h}^{p} \mathsf{m}'(h) \le C_1 \cdot \sum_{n = 1}^{\infty} n^{p-1} \mathsf{m}^{\Nb} \left( \left\{ \mathbf{g} \in \Gsf^{\Nb} : \tau(\mathbf{g}) > n \right\} \right).
   $$
   Note that $\mathsf{m}^{\Nb} \left( \left\{ \mathbf{g} \in \Gsf^{\Nb} : \tau(\mathbf{g}) > n \right\} \right)$ is the same as the probability that the random walk on $\Gsf / \Hsf \cong \Zb$ generated by $\pi_* \mathsf{m}$ starting from $0$ does not return to $0$ in $n$ steps, where $\pi : \Gsf \to \Gsf / \Hsf$ is the quotient map.  This probability is asymptotic to a constant multiple of $n^{-1/2}$ \cite[Proposition 4.2.4]{LL_RW_book}. Therefore, for some $C > 0$,
   $$
\sum_{h \in \Hsf} \abs{h}^{p} \mathsf{m}'(h) \le C \cdot \sum_{n = 1}^{\infty} n^{p-3/2}
   $$
   and the right hand side converges when $p < 1/2$.
\end{remark}

\subsection{An Example: fibered hyperbolic 3-manifolds} As in the introduction, let $\Hb^3$ denote real hyperbolic 3-space. The boundary at infinity $\partial\Hb^3$ is diffeomorphic to the two-sphere and $\Isom(\Hb^3)$ acts by diffeomorphisms. Let ${\rm Leb}$ denote a measure on $\partial\Hb^3$ induced by a smooth volume form. 

Fix a torsion-free cocompact lattice $\Gamma < \Isom(\Hb^3)$ such that the closed hyperbolic 3-manifold $M = \Gamma \backslash \Hb^3$ admits a fibration
$$
S \to M \to \Sb^1
$$
over the circle with a fiber $S \subset M$. Then we can view $\pi_1(S)$ as an infinite-index normal subgroup in $\Gamma$, with the quotient $\Ga / \pi_1(S) \cong \Zb$.
 
Using Proposition~\ref{prop:normal subgps} and work of Ballmann--Ledrappier, we will prove the following. 

\begin{proposition} \label{prop:fibration}
   There exists a probability measure $\mathsf m'$ with $\pi_1(S)=\ip{\supp \mathsf{m}'}$ whose associated stationary measure on $\partial \Hb^3$ is absolutely continuous with respect to ${\rm Leb}$. 

\end{proposition} 

\begin{proof} By ~\cite{BL_harmonic} there exists a symmetric probability measure $\mathsf m$ with $\Gamma=\ip{\supp \mathsf{m}}$ whose unique stationary measure $\nu$ is absolutely continuous with respect to ${\rm Leb}$ and where 
$$
\sum_{g \in \Gamma} \dist_{\Hb^3}(go,o) \mathsf{m}(g) < +\infty.
$$
Since $\Gamma$ acts cocompactly on $\Hb^3$, the \v Svarc--Milnor lemma implies that 
$$
\sum_{g \in \Gamma} \abs{g} \mathsf{m}(g) < +\infty,
$$
where $\abs{\cdot}$ is a word length on $\Gamma$ with respect to some fixed finite generating set. Since $\Gamma / \pi_1(S) \cong \Zb$, by Proposition~\ref{prop:normal subgps} there exists a probability measure $\mathsf m'$ with $\pi_1(S)=\ip{\supp \mathsf{m}'}$ whose associated stationary measure on $\partial \Hb^3$ is $\nu$, and hence absolutely continuous with respect to ${\rm Leb}$. 
\end{proof}

\end{document}